\documentclass[11pt]{amsart}
\usepackage{amsmath}
\usepackage{amsfonts}
\usepackage{amssymb}
\usepackage{tikz-cd}
\usepackage{mathtools}
\usepackage[numbers]{natbib}
\usepackage{theoremref}

\newtheorem{thm}{Theorem}[section]
\newtheorem*{thm*}{Theorem}

\newtheorem{cor}[thm]{Corollary}
\newtheorem*{cor*}{Corollary}

\newtheorem{lemma}[thm]{Lemma}
\newtheorem{prop}[thm]{Proposition}

\theoremstyle{definition}
\newtheorem{defn}[thm]{Definition}
\newtheorem{remark}[thm]{Remark}

\newtheorem{exam}[thm]{Example}

\newcommand\blfootnote[1]{%
  \begingroup
  \renewcommand\thefootnote{}\footnote{#1}%
  \addtocounter{footnote}{-1}%
  \endgroup
}

\newcommand{\calH}{\mathcal{H}}
\newcommand{\calK}{\mathcal{K}}

\newcommand{\bb}[1]{\mathbb{#1}}
\newcommand{\C}{\text{C}^*_r}

\newcommand{\acts}{\curvearrowright}

\newcommand{\normal}{\trianglelefteq}

\begin{document}

\title[]{Exactness vs C*-exactness for certain non-discrete groups}
\author[]{Nicholas Manor}

\maketitle

\blfootnote{Author supported by NSERC Grant Number 401226864.}

\begin{abstract}
It is known that exactness for a discrete group is equivalent to C*-exactness, i.e., the exactness of its reduced C*-algebra. The problem of whether this equivalence holds for general locally compact groups has recently been reduced by Cave and Zacharias to the case of totally disconnected unimodular groups.

We prove that the equivalence does hold for a class of groups that includes all locally compact groups with reduced C*-algebra admitting a tracial state. As a consequence, we present original proofs that totally disconnected locally compact (tdlc) invariant-neighbourhood (IN) groups and a class of groups introduced by Suzuki satisfy the equivalence problem, without using inner amenability.
\end{abstract}

\section{Introduction}

The property of exactness for locally compact groups has received much attention since it was first introduced and studied by Kirchberg and Wassermann in \cite{kirchberg,kw1,kw2}. In the setting of discrete groups, the theory of exactness is very well developed and connections to dynamics and coarse geometry have long been known \cite{delaroche1,ozawa}. It is also known for a discrete group $G$ that Kirchberg and Wassermann's original definition is equivalent to the apparently weaker condition (here called C*-exactness) that the reduced C*-algebra $\C(G)$ is exact \cite[Theorem 5.2]{kw1}.

More recently, dynamical \cite{delaroche1,bcl} and coarse geometric \cite{bcl} characterizations of exactness have been found for locally compact groups. However, it remains a major open problem to determine whether the equivalence of exactness and C*-exactness holds in general, and the problem has recently been reduced to the case of totally disconnected locally compact (tdlc) unimodular groups \cite{cz}. In other words, if there is an example of a non-exact locally compact group $G$ with exact reduced C*-algebra $\C(G)$, then there is necessarily a tdlc unimodular such group.

In this note we prove the following result (Theorem \ref{trace}), which is in analogy with one direction of Ng's characterization of amenability \citep[Theorem 8]{ng}, stating that a locally compact group $G$ is amenable if and only if $\C(G)$ admits a tracial state and is nuclear.

\begin{thm*}
Let $G$ be a locally compact group. If $\C(G)$ admits a tracial state and is exact, then $G$ is exact.
\end{thm*}

We give an example showing that the converse assertion, that if $G$ is exact then $\C(G)$ admits a tracial state and is exact, is not true. However, by weakening our hypotheses we get the following more powerful result (Theorem \ref{maintheorem}).

\begin{thm*}
Let $G$ be a locally compact group, and $(H_i)_{i \in I}$ be a family of open subgroups with the following conditions.
\begin{itemize}
\item For every $i$, there is a tracial state on $\C(H_i)$.
\item The union $\bigcup_{i \in I} H_i$ is equal to $G$.
\end{itemize}
If $G$ is C*-exact, then it is exact.
\end{thm*}

Previously, it was only known that exactness and C*-exactness are equivalent for groups with property (W) \citep[Theorem 7.3]{delaroche1}, and property (W) was later shown to be equivalent to inner amenability \citep[Theorem 3.5]{crannTanko}. We describe a class of non-inner amenable groups to which our results apply, hence expanding the class of groups for which the equivalence is known.

In section 2 we recall some basic definitions and results on exactness, and state some needed results on induced representations. In section \ref{main section} we prove the two results stated in the introduction. In section 4 we study examples, and in particular apply our tools to classes of tdlc unimodular groups. In subsection \ref{exactness of suzuki groups} we study a class of groups first considered by Suzuki \cite{suzuki}, where examples of non-discrete C*-simple groups were constructed. In subsection \ref{exactness for IN groups} we study tdlc groups which admit a conjugation invariant neighbourhood of the identity. These are precisely those tdlc groups whose reduced C*-algebras have non-zero center \cite[Corollary 1.2]{center}. We provide new structural results for this class, and resolve the equivalence problem for these groups without using inner amenability. For more results on invariant neighbourhood (IN) groups, see \cite{forrest2017existence}, \cite{iwasawa} and \cite{palmer}. In subsection 4.3 we show how to produce locally compact groups which are not inner amenable, but to which the methods in this paper apply. In subsection 4.4 we study a class of non-examples arising as automorphisms of infinite trees. These examples show that our results do not solve the equivalence problem for the rich class of locally compact groups of automorphisms of graphs.

\medskip

\textbf{Acknowledgments.} The author would like to thank Matthew Kennedy, Sven Raum, Nico Spronk, and the anonymous referee for helpful comments and stimulating discussions.

\section{Preliminaries}

\subsection{Exactness}

Recall \cite{kw1} that a locally compact group $G$ is called \textit{exact} if for every short exact sequence $0 \to I \to A \to B \to 0$ of $G$-C*-algebras, the corresponding sequence of reduced crossed products is also exact. That is to say,
\[0 \to I \rtimes_r G \to A \rtimes_r G \to B \rtimes_r G \to 0 \]
is an exact sequence of C*-algebras. Recall also that a C*-algebra $C$ is called exact if for every short exact sequence $0 \to I \to A \to B \to 0$ of C*-algebras, the corresponding sequence of spatial tensor products
\[ 0 \to I \otimes C \to A \otimes C \to B \otimes C \to 0 \]
is also exact. In the case where $I$, $A$ and $B$ are all equipped with the trivial action of $G$, the sequence of crossed products becomes
\[ 0 \to I \otimes \C(G) \to A \otimes \C(G) \to B \otimes \C(G) \to 0. \]
Hence, exact groups necessarily have exact reduced C*-algebra. We will say $G$ is \textit{C*-exact} if $\C(G)$ is an exact C*-algebra.

A priori, exactness appears to be strictly stronger than C*-exactness: there is no reason to expect the trivial action to capture all information about the exactness of $G$. Somewhat surprisingly, the two notions coincide for discrete groups \cite[Theorem 5.2]{kw1}.

\begin{thm}\thlabel{discrete}
If $G$ is a C*-exact discrete group, then it is exact.
\end{thm}

\subsection{Induced representations}

Recall that for a locally compact group $G$ and for any closed subgroup $H \le G$, for every continuous representation $\pi: H \acts \calH$ there is a canonical extension of $\pi$ to $G$, called the \textit{induced representation} and denoted by $\mathrm{Ind}_H^G(\pi): G \acts \calH \otimes L^2(G/H)$. See \cite[Appendix E]{bekka} for an overview.

\begin{exam}\label{example:induced representations}
Letting $\lambda_H$ denote the left regular representation of $H$, the induced representation $\mathrm{Ind}_H^G(\lambda_H)$ is unitarily equivalent to the left regular representation $\lambda_G$ of $G$.

Similarly, letting $1_H$ denote the trivial representation of $H$, the induced representation $\mathrm{Ind}_H^G(1_H)$ is unitarily equivalent to the left quasi-regular representation $G \acts L^2(G/H)$.
\end{exam}

The following result (\cite[Theorem F.3.5]{bekka}) is a well known permanence property of induced representations.

\begin{thm}\label{theorem:weak containment}
Let $G$ be a locally compact group and $H \le G$ a closed subgroup. Let $\pi : H \acts \calH$ and $\rho: H \acts \calK$ be two representations of $H$ and $\pi$ weakly contained in $\rho$. Then the induced representation $\mathrm{Ind}_H^G(\pi)$ is weakly contained in $\mathrm{Ind}_H^G(\rho)$.
\end{thm}

We next provide a proof for a very useful application of the above result.

\begin{cor}\label{corollary:amenable subgroup}
Let $G$ be a locally compact group and $H \le G$ a closed amenable subgroup. Then the left quasi-regular representation $G \acts L^2(G/H)$ is weakly contained in the left regular representation $\lambda_G$ of $G$.
\end{cor}

\begin{proof}
Since $H$ is amenable, it is a well known fact that its trivial representation $1_H$ is weakly contained in its left regular representation $\lambda_H$. The result now follows immediately follows from Example \ref{example:induced representations} and from Theorem \ref{theorem:weak containment}.
\end{proof}

\section{Groups with open amenable radical}\label{main section}

In the following sections, we show that the equivalence of exactness and C*-exactness holds for classes of locally compact groups which properly contain all discrete groups. The following observation, while simple, is the key to bootstrapping \thref{discrete} to the classes of groups considered in this note.

\begin{prop}\thlabel{openamenable}
Let $G$ be a locally compact group with an open normal amenable subgroup. If $G$ is C*-exact then it is exact.
\end{prop}

\begin{proof}
Suppose $G$ is C*-exact, and let $N \normal G$ be open and amenable. By Corollary \ref{corollary:amenable subgroup} the left quasi-regular representation $G \acts \ell^2(G/N)$ is weakly contained in the left regular representation of $G$.

Hence there is a surjective $\ast$-homomorphism $\C(G) \to \C(G/N)$, and exactness of C*-algebras passes to quotients \cite[Corallary 9.3]{wassermannexact}, so $G/N$ is also C*-exact. But $N$ is assumed to be open, meaning the quotient $G/N$ is a discrete group. \thref{discrete} then implies that $G/N$ is an exact group, and we already know that $N$ is exact as it is amenable. Since exactness is preserved by extensions \cite[Theorem 5.1]{kw2}, we conclude that $G$ is also exact.
\end{proof}

\begin{remark}
\thref{openamenable} can be strengthened so that $N$ is not necessarily open, but has the property that $G/N$ is exact if and only if it is C*-exact. Calling a group \textit{admissible} if it satisfies the equivalence, the statement can be strengthened as follows: extensions of amenable groups by admissible groups are admissible.
\end{remark}

Locally compact groups admitting a tracial state on the reduced C*-algebra have received recent attention in \cite{forrest2017existence,kennedyraumtraces,ng}. The following corollary relates the existence of a trace to the exactness of $G$, and is in analogy with the implication $(2) \implies (1)$ in Ng's characterization of amenability \cite[Theorem 8]{ng}.

\begin{thm}\thlabel{trace}
If $\C(G)$ admits a tracial state and is exact, then $G$ is exact.
\end{thm}

\begin{proof}
The main theorem of \cite{kennedyraumtraces} by Kennedy--Raum states that $\C(G)$ admitting a tracial state is equivalent to the existence of an open normal amenable subgroup in $G$. The result then follows from \thref{openamenable}.
\end{proof}

\begin{exam}
Exact groups do not necessarily admit a trace on their reduced C*-algebra. In particular, the converse of this theorem does not hold. Take for example $G = \text{SL}_2(\bb{R})$. Then \cite{kennedyraumtraces} implies $\C(G)$ does not have a trace since $G$ clearly does not admit an open amenable subgroup. However, we know by \cite[Theorem 6.8]{kw2} that connected groups are always exact, hence $G$ is exact.
\end{exam}

The author would like to thank Yemon Choi for pointing out that the nuclearity of $\C(\text{SL}_2(\bb{R}))$ also follows from the fact that $\text{SL}_2(\bb{R})$ is type I.

\medskip

The following theorem allows us to extend the class of groups to which our results apply.

\begin{thm}\thlabel{maintheorem}
Let $G$ be a locally compact group, and $(H_i)_{i \in I}$ a family of open subgroups with the following conditions.
\begin{itemize}
\item For every $i$, there is a tracial state on $\C(H_i)$.
\item The union $\bigcup_{i \in I} H_i$ is equal to $G$.
\end{itemize}
If $G$ is C*-exact, then it is exact.
\end{thm}

Since exactness is fundamentally a property of the ideal structure in reduced crossed products, we require the following technical fact about ideals in C*-algebras \cite[II.8.2.4]{blackadar}.

\begin{lemma}\thlabel{idealApproximation}
Let $A$ be a C*-algebra, and $(A_i)_{i\in I}$ a family of C*-subalgebras such that $\bigcup_{i \in I} A_i$ is dense in $A$. If $J$ is any closed ideal in $A$, then $J \cap \bigcup_{i \in I} A_i$ is dense in $J$.
\end{lemma}

We also note that if $0 \to I \to A \xrightarrow{q} B \to 0$ is an exact sequence of $G$-C*-algebras, then it is easy to see that $I \rtimes_r G \subseteq \ker q_G$. Indeed, take an arbitrary compactly supported function $f$ from the norm dense subset $C_c(G,I) \subseteq I \rtimes_r G$ and notice $q_G(f)(x) = q(f(x)) = 0$ for all $x \in G$ since $f(x) \in I = \ker q$. The inclusion $I \rtimes_r G \subseteq \ker q_G$ then follows by continuity of $q_G$. We now give the proof of \thref{maintheorem}.

\begin{proof}[Proof of \thref{maintheorem}]
Let $G$ and $(H_i)_{i\in I}$ be as in the theorem statement. Since each $H_i$ is an open subgroup we have the inclusion $\C(H_i) \subseteq \C(G)$, hence if $G$ is C*-exact then so is each $H_i$. But each $H_i$ admits a trace, thus \thref{trace} tells us that each $H_i$ is also exact. We show how exactness lifts up to $G$.

Let $0 \to I \to A \xrightarrow{q} B \to 0$ be an exact sequence of $G$-C*-algebras. Restricting each $G$-action to $H_i$, this is also exact as a sequence of $H_i$-C*-algebras. Since $H_i$ is an exact group, then
\[ 0 \to I \rtimes_r H_i \to A \rtimes_r H_i \xrightarrow{q_i} B \rtimes_r H_i \to 0 \]
is short exact. Using this fact, we will show that the corresponding sequence
\[ 0 \to I \rtimes_r G \to A \rtimes_r G \xrightarrow{q_G} B \rtimes_r G \to 0 \]
of $G$-crossed products is short exact, i.e., that $I \rtimes_r G = \ker q_G$.

\medskip

Since $\bigcup_{i \in I} H_i = G$, we have a norm dense inclusion $\bigcup_{i \in I} A \rtimes_r H_i \subseteq A \rtimes_r G$. We may then apply \thref{idealApproximation} to $\ker q_G$ to get $\overline{\bigcup_{i \in I} \ker q_i} = \ker q_G$. But the $H_i$ were shown to be exact, hence each $I \rtimes_r H_i = \ker q_i$, which gives us the inclusion $\ker q_G  \subseteq I \rtimes_r G$ as each $I \rtimes_r H_i \subseteq I \rtimes_r G$. The opposite inclusion $I \rtimes_r G \subseteq \ker q_G$ was discussed before the proof.
\end{proof}

\section{Examples and non-examples}

It was shown by Anantharaman-Delaroche in \cite{delaroche1} that exactness and C*-exactness are equivalent for inner amenable groups. Recall that a locally compact group $G$ is inner amenable if and only if there is a $G$-invariant state on $L^\infty(G)$.

In subsections 1 and 2 we show, without using inner amenability, that C*-exactness implies exactness for certain inner amenable groups. In subsection 3 we describe a class of non-inner amenable groups to which our results apply, therefore showing that our results expand the class of groups for which the equivalence problem is solved. In subsection 4 we show that certain automorphism groups of trees do not satisfy the hypotheses of \thref{maintheorem}.

\subsection{Exactness for Suzuki's groups}\label{exactness of suzuki groups}

The following result shows C*-exactness implies exactness for a class of tdlc groups considered by Suzuki in the context of C*-simplicity \cite{suzuki}. To avoid confusion, it should be noted that these are not the same as the Suzuki groups from finite group theory. The author would like to thank Yemon Choi for bringing this to his attention.

\begin{prop}\thlabel{SuzukiEquivalence}
Let $G$ be a locally compact group with a decreasing neighbourhood base $(K_n)_{n=1}^\infty$ of compact open subgroups, and an increasing sequence $(L_n)_{n=1}^\infty$ of open subgroups with the following additional properties.
\begin{itemize}
\item Each $K_n$ is a normal subgroup of $L_n$.
\item The union $\bigcup_{n=1}^\infty L_n$ is equal to $G$.
\end{itemize}
If $G$ is C*-exact, then it is exact.
\end{prop}

\begin{exam}\thlabel{semidirectsuzuki}
Suzuki describes a general construction in \cite{suzuki}. For each $n \in \bb{N}$, let $\Gamma_n$ be a discrete group and let $F_n$ be a finite group acting on $\Gamma_n$ by automorphisms.

We view the direct sum $\bigoplus_{n=1}^\infty \Gamma_n$ as a discrete group, and the direct product $\prod_{n=1}^\infty F_n$ as a compact group with the product topology. Defining the action $\prod_{n=1}^\infty F_n \acts \bigoplus_{n=1}^\infty \Gamma_n$ component-wise, the semidirect product $G := ( \bigoplus_{n=1}^\infty \Gamma_n ) \rtimes \prod_{n=1}^\infty F_n$ satisfies the conditions of \thref{SuzukiEquivalence}. To see this, for each $n \in \bb{N}$ set $K_n := \prod_{k=n+1}^\infty F_k$ and $L_n := (\bigoplus_{k=1}^n \Gamma_k) \rtimes \prod_{k=1}^\infty F_k $.
\end{exam}

\begin{proof}[Proof of \thref{SuzukiEquivalence}]
Since each $L_n$ has a compact (hence amenable) open normal subgroup $K_n$, it has a trace by \cite{kennedyraumtraces}. Hence if $G$ is C*-exact then it is exact by \thref{maintheorem}.
\end{proof}

The groups of \thref{SuzukiEquivalence} are inner amenable. Indeed, for each $n$ there is a conjugation invariant mean on $L^\infty(L_n)$ given by $\phi_n(f) = \int_{K_n} f$ since $K_n$ is normal in $L_n$. Extending the $\phi_n$ to $L^\infty(G)$ and picking a weak* cluster point of the sequence $(\phi_n)_{n \geq 1}$ produces a conjugation invariant mean on $L^\infty(G)$.

Since exactness of groups is preserved by closed subgroups and by extensions, and since exactness of C*-algebras is preserved by quotients, we have $G$ is exact if and only if each $L_n/K_n$ is exact. By \thref{SuzukiEquivalence}, we know this corresponds also to $\C(G)$ being exact. In the language of \thref{semidirectsuzuki}, this means $G$ is exact if and only if each $\Gamma_n$ is exact.

\subsection{Exactness for IN groups}\label{exactness for IN groups}

We now study exactness for tdlc groups with an additional topological property: we say a locally compact group $G$ is an \textit{IN group} (invariant neighbourhood group) if there is a compact neighbourhood $U$ of the identity which is invariant under conjugation, i.e., for every $g \in G$ we have $gU g^{-1} = U$. A closely related property is when $G$ admits a neighbourhood base at the identity consisting of conjugation invariant sets. In this case we say that $G$ is a \textit{SIN group} (small invariant neighbourhood group).

IN and SIN groups have a well understood structure theory with interesting applications to their representation theory. For further information on these groups, see \cite{forrest2017existence}, \cite{iwasawa} and \cite{palmer}.

\begin{remark}
It is easy to show that an IN group (resp. SIN group) $G$ is necessarily unimodular: letting $m$ denote the Haar measure on $G$, and setting $U$ to be an invariant compact neighbourhood of the identity, we have $gU = Ug$ for all $g \in G$. Hence $\Delta(g)m(U) = m(Ug) = m(gU) = m(U)$ for all $g$, implying the modular function $\Delta$ is constantly equal to 1.
\end{remark}

From this it is clear that IN groups, hence SIN groups, are inner amenable. To see this, fix a conjugation invariant compact neighbourhood $U \subseteq G$ of the identity, and define $\phi : L^\infty(G) \to \bb{C}$ by $\phi(f) = \int_U f$.

We will make use of the following structure theorem \cite[12.1.31]{palmer}, which strongly relates IN groups to SIN groups. It was originally proved by Iwasawa in \cite{iwasawa}.

\begin{thm}\thlabel{INquotient}
Let $G$ be an IN group. Then there is a compact normal subgroup $K \normal G$ so that $G/K$ is a SIN group with the quotient topology.
\end{thm}

\begin{remark}\thlabel{INasextension}
\thref{INquotient} tells us that IN groups are extensions of compact groups by SIN groups. In fact, the converse holds as well. That is to say, if $G$ has a compact normal subgroup $K \normal G$ so that $G/K$ is a SIN group, then $G$ is an IN group.

To see this, note that the map $q : G \to G/K$ is proper as the quotient is by a compact subgroup. So we may fix any invariant compact neighbourhood $U$ of the identity in $G/K$, and the preimage $q^{-1}(U)$ is an invariant compact neighbourhood of the identity in $G$.
\end{remark}

\begin{lemma}\thlabel{normalSubgroupSIN}
Let $G$ be a tdlc SIN group. Then $G$ admits a neighbourhood base at the identity consisting of compact open normal subgroups.
\end{lemma}

\begin{proof}
We will show that every compact open subgroup contains a compact open normal subgroup. Since the compact open subgroups form a neighbourhood basis this will complete the proof.

Let $K_1 \leq G$ be a compact open subgroup. Then there is a conjugation invariant neighbourhood $U$ of $e$ with $U\subseteq K_1$, and there is in turn a compact open subgroup $K_2 \subseteq U$. Since $U$ is conjugation invariant, we have the containment $\bigcup_{g \in G} gK_2 g^{-1} \subseteq U \subseteq K_1$, hence the open normal subgroup $K$ generated by the conjugates $\bigcup_{g \in G} gK_2 g^{-1}$ is also contained in $K_1$. As an open, hence closed, subgroup of a compact group, this implies $K \subseteq K_1$ is also compact.
\end{proof}

\thref{openamenable} together with \thref{normalSubgroupSIN} immediately give us the equivalence of exactness and C*-exactness for tdlc SIN groups.

\begin{cor}\thlabel{SINequivalence}
Let $G$ be a tdlc SIN group. If $G$ is C*-exact then it is exact.
\end{cor}

The main result of this section now follows quickly using structure theory of IN groups.

\begin{thm}\thlabel{IN groups}
Let $G$ be a tdlc IN group. If $G$ is C*-exact then it is exact.
\end{thm}

\begin{proof}
By \thref{INquotient}, $G$ has a compact normal subgroup $H$ so that $G/H$ is a SIN group with the quotient topology. Since $H$ is compact it is also amenable, so we get a surjection $\C(G) \to \C(G/H)$ as in the proof of \thref{openamenable}.

Since exactness passes to quotients \cite[Theorem 5.1]{kw2}, $G/H$ is also C*-exact. By \thref{SINequivalence} this implies $G/H$ is exact, hence $G$ is exact as an extension of a compact group $H$ by an exact group $G/H$.
\end{proof}

By a result of Losert \cite[Corollary 1.2]{center} the center of $\C(G)$ is non-trivial if and only if $G$ is IN. Therefore we obtain the following result with purely C*-algebraic hypotheses.

\begin{cor}
Let $G$ be tdlc. If $\C(G)$ is exact and has non-zero center, then $G$ is exact.
\end{cor}

Although IN groups comprise a rich class of unimodular groups, there are natural examples of tdlc unimodular groups which aren't IN.

\begin{exam}
Let $G := \text{SL}_2(\bb{Q}_p)$, where $p$ is any prime and $\bb{Q}_p$ denotes the set of $p$-adic rationals. Then $G$ is generated by its commutators, hence it is necessarily unimodular. However, $G$ is not IN: it is routine to show that all points other than $I_2$ and $-I_2$ can be conjugated arbitrarily far from $I_2$. Although this example is not IN, it is a linear group hence exact \cite{novikov}.
\end{exam}

The following result is of independent interest for the structure theory of totally disconnected IN groups.

\begin{prop}\thlabel{SINaslimit}
A topological group $G$ is tdlc and SIN if and only if there is an inverse system $(\Gamma_i, \phi_{ij})_{i>j \in I}$ of discrete groups so that $| \ker \phi_{ij} | < \infty$ for all $i > j$, and $G \cong \lim_\leftarrow \Gamma_i$.
\end{prop}

\begin{proof}
If $G$ is tdlc and SIN, there is a neighbourhood base of compact open normal subgroups by \thref{normalSubgroupSIN}. Ordering this normal neighbourhood base $(K_i)_{i \in I}$ by reverse inclusion gives us an inverse family $(G/K_i)_{i\in I}$ of discrete groups, with connecting maps $\phi_{ij} : G/K_i \to G/K_j$ whenever $i > j$. Since each $K_i$ is both compact and open, the $\phi_{ij}$ all have finite kernel, i.e., $| \ker \phi_{ij} | < \infty$. It is routine to check that $G$ is isomorphic to the inverse limit $\lim_\leftarrow G/K_i \subseteq \prod_i G/K_i$ with the relative product topology.

Suppose conversely that $G$ is isomorphic to an inverse limit $\lim_\leftarrow \Gamma_i$ with respect to the relative product topology, as described in the proposition statement. Viewing $G$ as a subgroup of $\prod_i \Gamma_i$, the subsets 
\[K_{j,k} := \{ (g_i)_{i\in I} \in \lim_\leftarrow \Gamma_i : g_j \in \ker \phi_{j,k} \}\]
with $j > k$ form a neighbourhood base of compact open normal subgroups.
\end{proof}

\begin{remark}
It is well known that compact tdlc groups are precisely the profinite groups \cite[Theorem 2.1.3]{profinite}, hence \thref{INasextension,SINaslimit} taken together tell us that tdlc IN groups are precisely extensions of profinite groups by inverse limits of discrete groups whose connecting maps have finite kernel.
\end{remark}

In the following example, we construct a tdlc IN group which is not SIN and non-exact.

\begin{exam}
Fix a finite group $F$ and a non-exact discrete group $H$ (interesting examples of non-exact groups may be found in \cite{mimura,osajda}). Then define $G = \prod_H F \rtimes H$, where $H$ acts by left translation on the compact group $\prod_H F$. It is clear that $G$ is an extension of the profinite group $\prod_H F$ by the discrete (hence tdlc SIN) group $H$, hence $G$ is a tdlc IN group.

However, it is not SIN: take any basic neighbourhood $U$ of the identity $e$ in $\prod_H F$, then some element of $U$ may be left-translated outside of $U$. This means that $U$ is not invariant under conjugation by elements in $H$.

Moreover, $G$ is non-exact as $H$ is a non-exact closed subgroup \cite[Theorem 4.1]{kw2}. This gives a tdlc IN group which is not SIN and non-exact.
\end{exam}

\subsection{Non-inner amenable groups}

In this subsection we outline conditions that allow us to construct a non-inner amenable group $G$ with an open amenable normal subgroup. The result of this section was inspired by \citep[Remark 2.6 (ii)]{forrest2017existence}. By \thref{openamenable} these groups are exact if and only if they are C*-exact, but since they are not inner amenable we cannot conclude this equivalence from Anantharaman-Delaroche's \citep[Theorem 7.3]{delaroche1}. This shows that our results strictly expand the class of groups for which the equivalence problem is solved.

\begin{prop}\thlabel{semidirect}
Let $N$ be an amenable locally compact group and $H$ a discrete group. If the only conjugation invariant mean on $H$ is evaluation at the identity, and if $\alpha : H \to \textrm{\emph{Aut}}(N)$ is an action such that there is no $\alpha(H)$-invariant mean on $N$, then the semi-direct product $N \rtimes H$ is not inner amenable.
\end{prop}

\begin{proof}
Suppose $\phi$ is a conjugation invariant mean on $N \rtimes H$. Since the only conjugation invariant mean on $H$ is evaluation at the identity, then $\phi$ must concentrate on $N$. Thus we may view $\phi$ as a mean on $N$ which is invariant under the action of $H$, a contradiction.
\end{proof}

Note that in the product topology on $N \times H$, the set $N \times \{e\} \cong N$ is open, hence it is an open amenable normal subgroup of $N \rtimes H$. Using \thref{openamenable}, it follows that the groups constructed in the above result are exact if and only if they are C*-exact.

\begin{exam}
It was proved in \cite[Remark 2.6 (ii)]{forrest2017existence} that $\mathbb{R}^2$ and $F_6 \subseteq SL_2(\mathbb{R})$ satisfy the hypotheses of \thref{semidirect}, and hence that $\mathbb{R}^2 \rtimes F_6$ is not inner amenable.

To produce a non-exact example, simply let $H$ be any non-exact discrete group and let $H$ act on $\mathbb{R}^2$ trivially. Then the semi-direct product $\mathbb{R}^2 \rtimes (F_6 \ast H)$ is non-exact, and it is also not inner amenable since by \cite[Theorem 1.1]{duchesne} the only conjugation invariant mean on $F_6 \ast H$ is evaluation at the identity.
\end{exam}

\subsection{Automorphism groups of trees}

For $d \geq 3$, we denote by $T_d$ the infinite $d$-regular tree, e.g., $T_4$ is the Cayley graph of the free group on two generators $\bb{F}_2$. The automorphism group $\text{Aut}(T_d)$ becomes a (non-discrete) tdlc group when equipped with the topology of pointwise convergence on the set of vertices $V(T_d)$. For every finite subset $S \subseteq V(T_d)$, the fixator $\text{Fix}_{\text{Aut}(T_d)}(S)$ of $S$ is a compact open subgroup. 

Fixing a vertex $b \in V(T_d)$ and an integer $n \geq 1$, we denote by $B_n(b)$ the ball centred at $b$ of radius $n$ in the path metric. The set of fixators $K_n := \text{Fix}_{\text{Aut}(T_d)}(B_n(b))$ forms a sequential neighbourhood base at the identity.

This subsection focuses on a rich class of subgroups of $\text{Aut}(T_d)$ with prescribed local actions first considered by Burger and Mozes \cite{burgermozes}. For a very good introduction to these groups and for additional results, we refer the reader to \cite[Section 4]{garrido}. We shall hereafter use the same notation and terminology as this source. In particular, for a subgroup $F$ of the symmetric group $S_d$ we will denote by $U(F)\leq \text{Aut}(T_d)$ the Burger-Mozes group of $F$.

More precisely, \thref{stronglyNonFree} shows that certain Burger-Mozes groups are too geometrically dense for our results of section \ref{main section} to apply. In Example \ref{example:alternating} we provide a concrete example of such a group. This shows that the equivalence problem for exactness remains unsolved for these groups.

\begin{defn}
We say a subgroup $G \leq \text{Aut}(T_d)$ is \textit{geometrically dense} if it does not fix any proper subtree, and does not fix any end in $\partial T_d$.
\end{defn}

\begin{prop}\thlabel{trivialAmenableRadical}
For any subgroup $F\leq S_d$, the amenable radical of $U(F)$ is trivial.
\end{prop}

\begin{proof}
Let $N$ be a non-trivial normal subgroup of $U(F)$. Since $U(\{ e\})$ is geometrically dense so is $U(F)$, hence $N$ is geometrically dense by \cite[Lemma 2.10]{garrido} as it is non-trivial and normal in $U(F)$. Using \cite[Lemma 2.9]{garrido} and applying the ping pong lemma, one can produce a closed copy of $\mathbb{F}_2$ in $N$, proving it is non-amenable.
\end{proof}

In particular, the amenable radical is open if and only if $U(F)$ is discrete. This result implies that the only Burger-Mozes groups admitting a tracial state are the discrete ones, which is precisely when the action $F \acts \{1,\dots,d\}$ is free \cite[Proposition 4.6 (v)]{garrido}.

\begin{cor}\thlabel{noTrace}
Let $F \leq S_d$ be a subgroup which does not act freely on $\{ 1, \ldots, d \}$, then the reduced C*-algebra $\C(U(F))$ does not admit a tracial state.
\end{cor}

\begin{proof}
By \thref{trivialAmenableRadical}, the amenable radical of $U(F)$ is trivial, and since $F \curvearrowright \{1,\ldots, d\}$ is not free then $U(F)$ is non-discrete. Hence the amenable radical is not open, and by \citep{kennedyraumtraces} this implies there is no tracial state on $\C(U(F))$.
\end{proof}

This means that \thref{trace} cannot be applied to this class. We would then like to determine whether we can write $U(F)$ as a union $\bigcup L_n$ of open subgroups with open amenable radical so that we may apply \thref{maintheorem}. 

Notice that if one of the $L_n$ is geometrically dense, then by the proof of \thref{trivialAmenableRadical} it has trivial amenable radical, hence does not have a trace if it is non-discrete. So if we would like to show that the $L_n$ cannot all have open amenable radical, then it suffices to show that at least one is geometrically dense.

\begin{remark}
For any subgroup $F \leq S_d$, the group $U(F)$ is compactly generated. Hence, if we write $U(F) = \bigcup L_n$ as an \textit{increasing} union of open subgroups $L_n \leq L_{n+1}$ then the sequence eventually terminates at some $L_N$. Since $U(F)$ is itself geometrically dense, this says that we can never write $U(F)$ as an increasing union of open subgroups which are not geometrically dense.
\end{remark}

The condition on the action $F \acts \{1,\ldots,d\}$ described in the following proposition is a strong converse to freeness.

\begin{prop}\thlabel{stronglyNonFree}
Let $F \leq S_d$ be a subgroup such that for every $l \in \{1,\dots,d\}$, the action of the stabilizer subgroup $\text{\emph{St}}_F(l) \acts \{1,\dots,l-1,l+1,\dots,d\}$ is transitive. If $U(F) = \bigcup L_n$ for some sequence $(L_n)_{n \geq 0}$ of open subgroups, then there is $n$ so that $L_n$ is geometrically dense. Hence $U(F)$ cannot be expressed as a union of open subgroups each having open amenable radical.
\end{prop}

We will need the following dynamical lemma.

\begin{lemma}\thlabel{minimal}
Let $F$ be as in \thref{stronglyNonFree}, and fix a half-tree $Y \subseteq T_d$. Then the fixator $\text{\emph{Fix}}_{U(F)}(T_d \setminus Y)$ acts minimally on $\partial Y$.
\end{lemma}

\begin{proof}
Let $b$ denote the root of $Y$, i.e., the unique vertex with degree $d-1$ in $Y$. Fix a legal labelling \cite[Section 4]{garrido} of $T_d$ so that, without loss of generality, the deleted edge at $b$ has the label $d$. We define a map $\phi : \partial Y \to \{ (j_n)_{n \geq 1} \in \{1 ,\dots, d\}^{\bb{N}} \ | \ j_1 \neq d, \ j_n \neq j_{n+1} \text{ for all } n \geq 1 \}$ by sending $x$ to the sequence $\phi(x) = (j_n)_{n \geq 1}$, where $j_n$ denotes the label of the $n^{\text{th}}$ edge along the geodesic ray $[b,x)$ joining $b$ to $x$.

The map $\phi$ is a bijection, and it is a homeomorphism when the codomain is equipped with the topology of point-wise convergence.

Now, given any two ends $x, y \in \partial Y$ with $\phi(x) = (j_n)_{n \geq 1}$ and $\phi(y) = (i_n)_{n \geq 1}$, we show how to produce a sequence $(g_n)_{n \geq 1}$ in $\text{Fix}_{U(F)}(T_d \setminus Y)$ such that $g_n \cdot y \to x$.

Since $\text{St}_F(d) \acts \{1,\dots,d-1\}$ is transitive, there is $h_1 \in \text{Fix}_{U(F)}(T_d \setminus Y)$ so that the first entry of $\phi(h_1 \cdot y)$ is $j_1$. Similarly, if $e_n$ is the $n^{\text{th}}$ edge along $[b,x)$, and if $\phi(x)$ and $\phi((h_n \cdots h_1) \cdot y)$ agree on the first $n$ entries, then by transitivity of $\text{St}_F(j_n) \acts \{1,\dots,j_{n}-1,j_n+1,\dots,d\}$ there is $h_{n+1} \in \text{Fix}_{U(F)}(T_d \setminus Y)$ so that $\phi(x)$ and $\phi((h_{n+1} \cdots h_1) \cdot y)$ agree on the first $n+1$ entries.

Setting $g_n = h_n \cdots h_1 \in \text{Fix}_{U(F)}(T_d \setminus Y)$, we then have convergence $g_n \cdot y \to x$. This proves that the action is minimal.
\end{proof}

\begin{proof}[Proof of \thref{stronglyNonFree}]
Fix a vertex $b$. Then for each $n \geq 1$, the complement $T_d \setminus B_n(b)$ is a disjoint union of finitely many (in fact $d(d-1)^{n-1}$) half-trees $Y_1, \dots , Y_k$. Note that $\partial Y_1, \dots, \partial Y_k$ form an open cover of $\partial T_d$.

Since the union $\bigcup L_n$ is equal to $ U(F)$, then there is some $N$ so that $L_N$ contains a hyperbolic element $h$, and since $L_N$ is assumed to be open, then it must contain the fixator $\text{Fix}_{U(F)}(B_n(b))$ for some $n \geq 1$.

Let $a_h \in \partial T_d$ be the attracting point of $h$, and without loss of generality assume it is in $\partial Y_1$. Let $x \in \partial Y_i$ be any other end, then by \thref{minimal} we may assume $x$ is not the repelling point of $h$. Hence we may find $m$ large enough that $h^m x \in \partial Y_1$ as $\partial Y_1$ is open, and again by \thref{minimal} we may send this end to $a_h$.

This proves the action of $L_N$ is transitive on $\partial T_d$, a similar argument shows that the set of ends arising as attracting points of hyperbolic elements in $L_N$ is all of $\partial T_d$. Since any $L_N$-invariant subtree must contain the axes of all hyperbolic elements, then there are no proper $L_N$-invariant subtrees. This proves $L_N$ is geometrically dense.
\end{proof}

The following is a concrete example of a group satisfying the hypotheses of \thref{stronglyNonFree}.

\begin{exam}\label{example:alternating}
For $d \geq 4$, the alternating subgroup $A_d \leq S_d$ satisfies the hypotheses of the previous result. Indeed, fix $l \in \{1,\dots,d\}$ and pick distinct $i,j \in \{1,\dots, l-1, l+1, \dots ,d\}$. Then, since $d$ is at least 4, there is $k \neq i,j,l$ hence the 3-cycle $(ijk)$ fixes $l$ and is in $A_d$. Moreover, $(ijk)$ sends $i$ to $j$, proving that the point stabilizer of $A_d$ at $l$ is transitive.

By \thref{stronglyNonFree}, $U(A_d)$ cannot be written as a union of open subgroups each having open amenable radical. This means that \thref{maintheorem} does not apply to $U(A_d)$.
\end{exam}

\bibliography{references}
\bibliographystyle{amsplain}

\end{document}